\documentclass[12pt]{amsart}

\usepackage[dvips]{graphics}
\usepackage{epsfig}
\usepackage{amssymb}
\usepackage{amsthm}
\usepackage[mathscr]{eucal}
\usepackage{enumitem}
 \usepackage[usenames,dvipsnames]{pstricks}
 \usepackage{epsfig}

\usepackage{color}

\newtheorem{Theorem}{Theorem}[section]
\newtheorem{theorem}[Theorem]{Theorem}

\newtheorem{proposition}[Theorem]{Proposition}

\newtheorem{lemma}[Theorem]{Lemma}

\newtheorem{fact}[Theorem]{Fact}

\newtheorem{claim}[Theorem]{Claim}

\theoremstyle{definition}

\newtheorem{example}[Theorem]{Example}
\newtheorem{Remark}[Theorem]{Remark}
\newtheorem{remark}[Theorem]{Remark}
\newtheorem{remark/def}[Theorem]{Remark/Definition}

\newtheorem{definition}[Theorem]{Definition}

\newsavebox{\indbin}
\savebox{\indbin}{\begin{picture}(0,0)
\newlength{\gnu}
\settowidth{\gnu}{$\smile$} \setlength{\unitlength}{.5\gnu}
\put(-1,-.65){$\smile$} \put(-.25,.1){$|$}
\end{picture}}
\def \indo {\mathop{\smile \hskip -0.9em ^| \ }}

\newcommand{\be}{\begin{enumerate}}
\newcommand{\bd}{\begin{defn}}

\newcommand{\bt}{\begin{theorem}}
\newcommand{\bl}{\begin{lemma}}
\newcommand{\ee}{\end{enumerate}}
\newcommand{\ed}{\end{defn}}
\newcommand{\et}{\end{theorem}}
\newcommand{\el}{\end{lemma}}

\newcommand{\la}{\langle}
\newcommand{\ra}{\rangle}

\newcommand{\ov}{\overline}
\newcommand{\cd}{\centerdot}

\newcommand{\CF}{{\mathcal F}}
\newcommand{\CG}{{\mathcal G}}

\newcommand{\CM}{{\mathcal M}}

\newcommand{\CS}{{\mathcal S}}

\newcommand{\id}{\operatorname{id}}
\newcommand{\Aut}{\operatorname{Aut}}
\newcommand{\aut}{\operatorname{Aut}}

\def\eq{\operatorname{eq}}

\def\dcl{\operatorname{dcl}}

\def\acl{\operatorname{acl}}

\def\tp{\operatorname{tp}}

\def\ul{\underline}

\def\ob{\operatorname{Ob}}
\def\mor{\operatorname{Mor}}

\def\init{\operatorname{init}}
\def\ter{\operatorname{ter}}

\def\bsigma{\mbox{\boldmath $\sigma$}}
\def\btau{\mbox{\boldmath $\tau$}}

\def\bmu{\mbox{\boldmath $\mu$}}

\title[Non-commutative groupoids]{Non-commutative groupoids obtained from the failure of $3$-uniqueness in stable theories}
\author{Byunghan Kim}
\author{SunYoung Kim}
\author{Junguk Lee}

\address{Department of Mathematics\\ Yonsei University\\
50 Yonsei-Ro, Seodaemun-Gu\\
Seoul 120-749, Korea}

\email{bkim@yonsei.ac.kr}
\email{sy831@yonsei.ac.kr}
\email{ljw@yonsei.ac.kr}

\thanks{All authors were supported by Samsung Science Technology Foundation under Project Number SSTF-
BA1301-03.}
\begin{document}

\begin{abstract}
We construct a possibly non-commutative groupoid from the failure of $3$-uniqueness of a strong
type. The commutative groupoid constructed by John Goodrick and Alexei Kolesnikov  in \cite{GK} lives in the center
 of the groupoid.

A certain automorphism group  approximated by the vertex groups of the non-commutative groupoids  is suggested as a ``fundamental group" of the strong type.
\end{abstract}

\maketitle

\section{Introduction}

%The intension of this last section is to find an analog  of the fact that homology group is an abelianization of %homotopy group.  As

In singular homology theory, one of the differences between the fundamental group and the first homology group $H_1$  is that the former is not necessarily commutative while  the latter is. In the earlier papers \cite{GK},\cite{gkk},\cite{GKK} by Goodrick,  Kolesnikov (and the first author),  an analogue of  homotopy/homology theory is developed in the context of model theory but where the ``fundamental group'' introduced is always commutative. In this paper, by taking an approach closer to the original idea of homotopy theory, we suggest how to construct a different fundamental group in a non-commutative manner. More precisely, from a symmetric witness to the failure of $3$-uniqueness in a stable theory,
 we construct a new groupoid $\CF$ whose ``vertex groups'' $\mor_{\CF}(a,a)$ need not be  abelian. In fact, we will show that $\mor_{\CG}(a,a) \leq  Z(\mor_{\CF}(a,a))$, where $\CG$ is the commutative groupoid constructed in \cite{GK} and \cite{gkk}.  We may call $\CF$ a {\em non-commutative groupoid } constructed from the symmetric witness. But unlike the groupoid $\CG$, this new groupoid $\CF$ is definable only in certain cases (e.g. under $\omega$-categoricity); in general, it is merely invariant over some  set.

%[[summary each sections??]]

\medskip

We work in  a complete \emph{stable} theory $T$ with a fixed
monster model $\CM=\CM^{\eq}$.
Unless said otherwise,  tuples are  from $\CM$ and
sets $A,B,\ldots$ are small subsets of $\CM$; and there is an independence notion among sets, defined by nonforking.
For tuples $a_0,a_1,\dots$, we write $a_{01}$ to denote $a_0a_1$ and so on. Throughout this paper {\em we also fix an algebraically closed  set $A$ and a complete type $p$ of possible infinite arity over $A$}. For a tuple $c$,
$\ov{c}$ denotes $\acl(cA)$. If $\{a,b,c\}$ is an $A$-independent set of realizations of $p$, then we let
$$\widetilde{ab}:=\dcl(\ov{ac}\ov{bc})\cap \ov{ab}.$$ Due to stationarity, this set only depends on $a$ and $b$.
 The rest
 notational convention we take is standard.
For example, $a\equiv_Bb$ means $\tp(a/B)=\tp(b/B)$; and
 $\Aut(C/B)$ is the group of
elementary maps from $C$ {\em onto} $C$ fixing $B$ pointwise.
In addition, $\Aut(\tp(f/B))$ means $\Aut(Y/B)$ where $Y$ is the solution set of
$\tp(f/B)$.

Now we recall definitions of notions which we will use throughout.
A {\em groupoid} is a category where every morphism is invertible. Hence in a groupoid,
for each object $a$, $\mor(a,a)$ forms a group called a {\em vertex group}. If all the vertex groups are abelian we
call the groupoid {\em abelian} or {\em commutative}.
We say a groupoid is {\em connected} ({\em finitary}, resp.) if
for any two objects $a,b$, $\mor(a,b)$ is non-empty  (finite, resp.).  If a groupoid is connected then
each vertex group is isomorphic.

 Originally, $3$-uniqueness is defined functorially in \cite{Hr},\cite{gkk}, but as we will not use amalgamation notion the following equivalent definition would suffice in this note.

\begin{definition} \cite{gkk} We say the fixed complete type $p$ has {\em $3$-uniqueness over $A$}
if whenever  $\{a_0,a_1,a_2\}$ is an $A$-independent set of realizations of $p$, and for $0\leq i<j\leq 2$, $\sigma_{ij}\in \aut(\ov{a_{ij}}/\ov{a_i}\ \ov{a_j})$, then $\sigma_{01}\cup \sigma_{02}\cup \sigma_{12}$ is also an elementary map.
\end{definition}

\begin{fact} \cite{Hr},\cite{gkk} Let $a,b,c\models p$ be independent over $A$.
Then $p$ has $3$-uniqueness over $A$  iff $\widetilde{ab}=\dcl(\ov{a}\ov{b})$.
%where  $\widetilde{ab}:=\dcl(\ov{ac}\ov{bc})\cap \ov{ab}$.
\end{fact}

We now recall a certain automorphism group which plays a role of the (abelianized) fundamental group of $p$ in the homotopy theory
 of model theory introduced   in \cite{GK},\cite{GKK}.

\begin{definition} \cite{gkk},\cite{GKK} Let $\{a,b,c\}$ be $A$-independent set of realizations of $p$.
We let
%$$\widetilde{ab}:=\dcl(\ov{ac}\ov{bc})\cap \ov{ab},$$
%and let
$$\Gamma_2(p):=\aut(\widetilde{ab}/\ov{a}\ov{b}).$$
\end{definition}

Since the homology groups of $p$ will  not be dealt with,  we do not recall the definition of those,
but  only point out the following  proved in \cite{GKK}.

\begin{fact}\label{Hure} The group
$\Gamma_2(p)$ is profinite abelian and
 isomorphic to the type's 2nd homology group $H_2(p)$.
\end{fact}

There indeed is a mismatch in numbering. The group $\Gamma_2(p)$ corresponds to the fundamental group
$\pi_1$ (or its abelianization), and so  should do $H_2(p)$ to the first homology group  in algebraic topology.
An higher dimensional version of Fact \ref{Hure} is proved in \cite{GKK1}.

Our goal in this paper is to introduce a possibly non-commutative ``fundamental group"  $\Pi_2$ of $p$,
in which $\Gamma_2(p)$ places in the center. In section 2, we give a motivational example. Namely the model of a
connected groupoid having a  given vertex group $G$.   It turns out that $\Pi_2(p)=G$ and
$\Gamma_2(p)=Z(G)$ if we take $p$ as the $1$-type of any object.

In section 3, we develop a general theory for constructing our desired non-commutative connected finitary groupoid from a symmetric witness  to the non-$3$-uniqueness of $p$.

In section 4, we show that $\Pi_2(p)$, a certain automorphism group partially approximated by
the vertex groups of non-commutative groupoids constructed from the failure of $3$-uniqueness of $p$, is a normal subgroup of $\aut(\widetilde{ab}/\ov{a})$  where $a,b\models p$ are $A$-independent, and
$\Gamma_2(p)$ is central in $\Pi_2(p)$.

\section{Finitary groupoid examples}

 Let $G$ be an arbitrary finite group.  Now let $T_G$ be the complete
stable  theory of the connected  finitary groupoid  $(O,M,.,\init,\ter)$ with
the standard setting. Namely the sorts $O, M$ represent  the {\em infinite} sets of all objects and  morphisms, respectively;  $.$ is the composition map between
morphisms; and $\init,\ter:M\to O$ are maps indicating initial and terminal objects, respectively, of a morphism. Moreover   $G_a:=\mor(a,a)$ is isomorphic to $G$  for any $a\in O$. Now due to weak elimination of imaginaries,
$\emptyset=\acl^{\eq}(\emptyset),$ and we let the $p(x)$ be the unique $1$-type over $\emptyset$ of any object.

%As pointed out in \cite{GK},

\begin{remark} \label{gpdauto} \cite[4.2]{GK} Fix $a\in O$, and for $u(\ne a)\in O$, choose $g_u\in G_u$. Then the following map $\sigma$ is a structure automorphism of the groupoid:
\be\item
$\sigma$ is the identity map on $O$, and on $G_a$;
\item
for $u(\ne a)\in O$, we have $\sigma(f)=g_u.f$ or $=f.g^{-1}_u$, if $f\in \mor(a,u)$, or $\in \mor(u,a)$, respectively;
and
\item
for $u,v(\ne a)\in O$ and $f\in \mor(u,v)$, we have  $\sigma(f)=g_v.f.g^{-1}_u.$
\ee
\end{remark}

 We fix distinct $a,b\in O$ and a morphism $f_{0}\in \mor(a,b)$.
Due again to weak elimination of imaginaries it follows that $\widetilde{ab}$, $\ov{ab}$, and $\mor(a,b)G_aG_b$ are all interdefinable, and
 $$\Gamma_2(p)=\Aut(\widetilde{ab}/\ov a,\ov
 b)=\Aut( \mor(a,b)/aG_abG_b).$$
Hence indeed
(note that
%by Remark \ref{gpdauto},
$\mor(a,b)\subseteq  \dcl(f_{0}G_a)$)
$$\Gamma_2(p) =\Aut(X/aG_abG_b).$$
where $X$ is the finite solution set
 of $\tp(f_{0}/aG_abG_b)$.

Now  for $f\in X$ there is unique $x\in G_a$ such that $f=f_0.x$,
and we claim that this $x$ must be in $Z(G_a)$.

\begin{claim}\label{centermor}
For $x\in G_a$, we have $g:=f_0.x\in X$ iff $x\in Z(G_a)$.
 \end{claim}
\begin{proof}
($\Rightarrow$) Since $g\in X$, $f_0\equiv_{G_aG_b} g$. Then for any
$y\in G_a$, we have $$f_0.y.f^{-1}_0(\in
G_b)=g.y.g^{-1}=f_0.x.y.x^{-1}.f^{-1}_0.$$
 Hence $x\in Z(G_a)$.

($\Leftarrow$) There is $z\in G_b$ such that $f_0=z.g$. Now since
$x\in Z(G_a)$, for any $y\in G_b$ we have
$$g^{-1}.y.g.x^{-1}=f_0^{-1}.z.y.z^{-1}.f_0.x^{-1}=x^{-1}.f_0^{-1}.z.y.z^{-1}.f_0=g^{-1}.z.y.z^{-1}.g.x^{-1}.$$
Hence $y=z.y.z^{-1}$, i.e. $z\in Z(G_b)$. Hence then  by Remark \ref{gpdauto} there is an automorphism fixing $aG_abG_b$
 pointwise while sending $f_0$ to $g$. Hence $g\in X$.
\end{proof}

 \begin{claim}
$\Gamma_2(p)=Z(G)$.
 \end{claim}
\begin{proof}
The proof will be similar to that of Proposition 2.22 in \cite{gkk}.
Note firstly that due to Claim \ref{centermor}, $Z(G_a)$ acts  on $X$ as an
obvious manner. This action is clearly regular. Secondly
$\Aut(X/aG_abG_b)$ also regularly  acts on $X$. Moreover since each $\sigma \in \Aut(X/aG_abG_b)$ fixes
$G_a$ pointwise,  it clearly follows  that the two actions commute. Hence
they are the same group.
\end{proof}

Due to Remark \ref{gpdauto}, we have   $f\equiv_{aG_ab} f_{0}$ for any $f\in \mor(a,b)$, i.e.,
 $\mor(a,b)$ is the solution set of
$\tp(f_0/aG_ab)$ or $\tp(f_0/\ov{a} b)$. Moreover for $f\in \mor(a,b)$, it follows $$\dcl(f\ov{a})=\dcl(f_0\ov{a})=\dcl(\mor(a,b),\ov{a})=\dcl(\mor(a,b),\ov{a}\ov{b})=\ov{ab}.$$
 %and $f$ are interdefinable over $\ov a$.
 Hence,  $$\aut(\widetilde{ab}/\ov{a})=\aut(\mor(a,b)/\ov{a})=\aut(\tp(f_0/\ov{a}b)).$$
We further claim the
following.

\begin{claim}\label{iso}
$G$ is isomorphic to $\Aut(\mor(a,b)/\ov{a})=\Aut(\mor(a,b)/G_a)$.
Hence $\Gamma_2(p)=  \Aut( \mor(a,b)/G_aG_b)=   Z(\Aut(\widetilde{ab}/\ov{a})).$
\end{claim}
\begin{proof}
We know $G$ and $G_b$ are isomorphic. Now given
$\sigma\in G_b$ we assign  an automorphism  $\bsigma\in\Aut(\mor(a,b)/\ov{a})$ sending
 $f(\in \mor(a,b))\mapsto \sigma.f$.
This mapping is well-defined, since if $g\in \mor(a,b)$ so that $g=f.\mu$ for some $\mu\in G_a$, then
$\bsigma(g)=\bsigma(f).\mu=\sigma.f.\mu=\sigma.g$.
 %This is well-defined mapping since $\bsigma$ fixes $G_a$ pointwise.
Now this correspondence
is clearly 1-1 and onto. It is obvious that
the correspondence is a group isomorphism.
\end{proof}
In the following section we try to search this phenomenon
 in the general stable theory context.   Namely given the abelian groupoid
 built  from a  {\em symmetric witness}
  introduced in \cite{GK}, we construct an extended groupoid
  possibly non-abelian but the abelian groupoid places in the
  center of the new groupoid.  In the case of
above $T_G$, as we seen the morphism group of the abelian groupoid
is $Z(G)$, but in the extended one the morphism group is equal to
$G$.

%[[gpoid def. connected , finitary]]

\section{The non-commutative  groupoid $\CF$}

We recall the notion of symmetric witnesses introduced in \cite{GK} or \cite{GKK}.

\begin{definition}\label{full_symm_witness}
A \emph{(full) symmetric witness to non-$3$-uniqueness} (over the algebraically closed  set $A$) is a tuple $(a_0, a_1, a_2, f_{01}, f_{12}, f_{02}, \theta(x,y,z))$ such that $a_0, a_1, a_2$ and $f_{01},f_{12}, f_{02}$ are finite tuples,
$\{a_0,a_1,a_2\}$ is independent over $A$, $\theta(x,y,z)$ is a formula over $A$, and:
\begin{enumerate}
\item
$a_{ij}\subset f_{ij} \in \overline{{a}_{ij}}\smallsetminus \dcl(\overline{{a}_i}, \overline{{a}_j}) $;
%\item
% $a_{ij}$ is a subtuple of $f_{ij}$;
\item
$a_{01}f_{01} \equiv_A a_{12}f_{12} \equiv_A a_{02}f_{02}$;
\item
$f_{01}$ is the unique realization of
$\theta(x, f_{12}, f_{02})$, and so are  $f_{12},f_{02}$  of $\theta(f_{01}, y, f_{02})$, $\theta(f_{01}, f_{12}, z)$, respectively; and
\item each
$\tp(f_{ij}/\ov{{a}_i}\ \ov{{a}_j})$ is isolated by $\tp(f_{ij}/a_{ij}A)$.
\end{enumerate}
\end{definition}

The following (proved in \cite{GK}) is the key technical point saying that we have ``enough'' symmetric witnesses:

\begin{proposition}\label{full symm}
If $(a'_0, a'_1, a'_2)$ is the beginning of a Morley sequence of finite tuples over $A$ and $f'$ is a finite tuple in $\widetilde{a'_{01}}\smallsetminus \dcl(\ov{a'_0}, \ov{a'_1})$, then there is some full symmetric witness $(a_0, a_1, a_2, f, g, h, \theta)$ such that $f' \in \dcl(fA)$ and $a'_i \in \dcl(a_iA) \subseteq \ov{a'_i}$ for $i = 0, 1, 2$.

Hence if the complete type  $p$ does not have $3$-uniqueness over $A$, then there is  a  symmetric witness $(a_0,a_1,a_2,\dots)$ over $A$ such that $a_i\in \ov{c_i}$ for some $A$-independent realizations $c_0,c_1,c_2$ of $p$.
\end{proposition}

%We will work with  symmetric witnesses to the failure of $3$-uniqueness as defined in Definition~\ref{full_symm_witness} %above.

From now on for notational simplicity, we suppress $A$ to $\emptyset$ (by naming the set).
 %so that $\emptyset =\acl(\emptyset)$.
We fix some more notations that we will refer to throughout the rest.  Fix a symmetric witness $W=(b_0, b_1, b_2, f'_{01}, f'_{12},f'_{02},\theta)$ to the failure of $3$-uniqueness  over $\emptyset = \acl(\emptyset)$. We put  $p(x)=\tp(b_0)$. The following facts are shown in
\cite{GK},\cite{gkk} (see also \cite{GKK},\cite{GKK1}).

\begin{definition}\label{gengpoid} By    a {\em generic abelian groupoid in} $p$, we mean   an $\emptyset$-type-definable connected finitary  groupoid  $\CG$ such  that
\be\item
$\ob(\CG)=p(\CM)$,  $\mor (\CG)$ is $\emptyset$-type-definable, and
maps $$\init,\ \ter:\mor (\CG)\to\ob(\CG)$$ indicating initial and terminal objects of a morphism respectively are
$\emptyset$-definable, and so is the composition map between morphisms;

\item
 for $f\in \mor_{\CG}(b_0,b_1)$,
$$\mor_{\CG}(b_0,b_1)=\{g|\ g\equiv_{b_{01}}f\}=\{g|\ g\equiv_{\ov{b_0}\ \ov{b_1}}f\}\mbox{;\ and}$$

%and there lives the canonical group $G$ in $\acl(\emptyset)$

\item for each $a\models p$, the vertex group $\mor_{\CG}(a,a)$ is finite  abelian.
\ee
\end{definition}

\begin{fact}\label{gpoidfromwitness} From the witness $W$,  we can construct an $\emptyset$-type-definable generic abelian  groupoid $\CG$ in $p$ such  that
\be
\item
there exists $b_{01}$-definable  bijection  $$\pi_{01}:\{f'|\  f'\equiv_{b_{01}}f'_{01}\}=\{f'|\  f'\equiv_{\ov{b_{0}}\ \ov{b_{1}}}f'_{01}\}\to \mor_{\CG}(b_0,b_1),\mbox{ and}$$
%and there lives the canonical group $G$ in $\acl(\emptyset)$
\item $\theta$ represent the composition, i.e., for $f_{ij}:=\pi_{ij}(f'_{ij})\in\mor_{\CG}(b_i,b_j)$, we have
%for the morphism in  $\mor(b_i,b_j)$ corresponding to  $f'_{ij}$ by the bijection, then
$f_{12}. f_{01}=f_{02}$.
\ee
\end{fact}

\begin{Remark}\label{bdgp}
We also fix such $\CG$,  the {\em  groupoid obtained from the witness} $W$. For $a, b\models p$, for convenience, $X_{ab}$ denotes $\mor_{\CG}(a,b)$, and $X_a=X_{aa}$ denotes the
vertex group $\mor_{\CG}(a,a)$. Hence if $a,b$ are independent then $X_{ab}$ is the solution set of
$\tp(f/ ab)$ (so of $\tp(f/\ov a\ov b)$) where $f\in X_{ab}$.

As is known,
$\eta:X_a\to X_b$ sending $\sigma$ to $f\circ \sigma\circ f^{-1}$ for some (any) $f\in X_{ab}$ is a canonical group isomorphism. Hence there lives a finite abelian group $G=(\bigcup\{X_a \mid  a\in \ob(\CG)\})/\sim$ (where $\sigma\sim \sigma'$ if $\eta (\sigma ) = \sigma'$) in $\acl(\emptyset)$, which is  canonically isomorphic to each group $X_a$.
We call $G$ the \emph{binding group of the groupoid $\CG$}. (Hence in the example in section 2, if the finite group $G$ there is abelian, then it lives in $\acl(\emptyset)$ as the binding group, so if we name $\acl(\emptyset)$, then for any $a\in O$,
$\dcl(a)=\acl(a)=\acl(aG_a)$. This need not hold if $G$ there is not abelian.)

The group $(G,.)$ naturally acts on the set $\mor(\CG)$. For $f\in X_{ab}$ and $\bsigma\in G$, the ($\emptyset$-definable) left action $\bsigma.f $ is given by the composition $\sigma. f$, where $\sigma$ is the unique element in $\bsigma\cap X_b$; the right action $f.\bsigma$ is given similarly. But the two actions are equal. Namely, for all $f\in X_{ab}$, $g\in X_{bc}$, for all $\bsigma,\btau\in G$ we have
$\bsigma. f=f. \bsigma$;
$(g.f). \bsigma = g. (f.\bsigma)$; and
$f. (\bsigma. \btau)=(f.\bsigma).\btau$. Clearly this action on $X_{ab}$ is regular, and $|G|=|X_{ab}|$.
\end{Remark}

We fix more notations. Given independent $a,b\models p$ and $f_{ab}\in X_{ab}$ (so
 $b_{01}f_{01}\equiv abf_{ab}$), we write $G_{ab}$ to denote $\Aut(\tp(f_{ab}/\ov a\ov b))=\Aut(\tp(f_{ab}/a b))$.
 Note that $G_{ab}=\Aut(\tp(f'_{ab}/\ov{a} \ov{b}))=\Aut(\tp(f'_{ab}/a b))$ too, where $b_{01}f'_{01}\equiv abf'_{ab}$, since $\dcl(f_{ab})=\dcl(f'_{ab})\ni a,b$.

\begin{fact}\label{bdgpiso}
 Let  $a,b \models p$ be independent.
 \be
\item If we let  $(f_0,g_0)\sim(f_1,g_1)\in X_{ab}^2$ when there is $\bsigma\in G$ such that $g_j=\bsigma. f_j$ ($j=0,1$), then  $\sim$ is an equivalence relation on  $X^2_{ab}$ and
 the map  $[\ ]:X_{ab}^2/\sim\to G$ sending $[(f_j,g_j)]$ to $\bsigma$ is the unique bijection such that
for
$f,g,h\in X_{ab}$,
$$[(f,g)].[(g,h)]=[(f,h)].$$

\item
For any $f,g\in X_{ab}$ and $\bsigma\in G$, we have $\dcl(f)=\dcl(g)$  and $\tp(f,\bsigma.f)=\tp(g,\bsigma. g)$.

\item There exists the canonical isomorphism  $\rho_{ab}:G\to G_{ab}$  sending $\bmu\in G$
to $\mu\in G_{ab}$ such that $\bmu. f= \mu(f)$ for some (any) $f\in X_{ab}$.
In other words $G$ again uniformly and canonically binds all the groups of the form $G_{ab}$ with independent $a,b\models p$.

\ee
\end{fact}
\begin{proof}
(1) This easily follows from the regularity of the action of $G$ on $X_{ab}$.

(2) comes from that $G\subseteq \acl(\emptyset)$.

(3) Here the abelianity of $G$ is used. It needs to show that $\bmu. f= \mu(f)$ implies
 $\bmu. g= \mu(g)$, for  any $f,g\in X_{ab}$ and $\bmu \in G$, $\mu \in G_{ab}$. Now there is
 $\sigma \in G_{ab}$ such that $g=\sigma (f)$. Then from (2), or directly,  $\sigma(\bmu.f)=\bmu.\sigma(f)=\bmu.g$.  Thus $\mu(g)=\mu\circ \sigma(f)=\sigma\circ \mu(f)=\sigma(\bmu.f)=\bmu.g.$ The rest can easily be checked.
\end{proof}

Now we are ready to extend the construction method given in \cite{GK} to find
another groupoid $\CF$ (which in general is  not  $\emptyset$-type-definable but $\emptyset$-invariant)
from the fixed symmetric witness $W$. The class
$\ob(\CF)$ of objects will be the same as $\ob(\CG)$.  But  $\CF$ need not be
abelian as  the vertex group $\mor_{\CF}(b_i,b_i)$ will be isomorphic to  $\Aut(Y_{01}/\ov b_0)$, where $Y_{01}$ is the
{\em possibly infinite} set
$$Y_{b_{01}}=Y_{01}:=\{f\in\dcl(f_{01},\ov{b_0})|\ f\equiv_{\ov{b_0}}f_{01}
\text{ and } \dcl(f\ov{b_0})=\dcl(f_{01}\ov{b_0})\}.$$ Note that
$\dcl(f_{01},\ov{b_0})=\dcl(f_{01}b_1\ov{b_0})$ since $b_1\in\dcl(f_{01})$.
Moreover  $Y_{01}$ and  $Y'_{01}$, the set defined the same way as $Y_{01}$ but substituting $f'_{01}$ for $f_{01}$, are interdefinable. Furthermore, we shall see that $\mor_{\CG}(b_i,b_i)\leq
Z(\mor_{\CF}(b_i,b_i))$ (Claim \ref{centergp}). We will call $\CF$
a {\em non-commutative  groupoid} constructed from the symmetric witness $W$.

%[[What if replacing $\ov{b_0}$ by simply $b_0$ in defining $Y_{01}$?? but below 3.7 says both are equal?? or that is %obvious?? ]]

\begin{remark}
The set $Y_{01}$ defined above depends only on $b_0$ and $b_1$ and not on the choice of $f_{01} \in X_{b_{01}}$.
\end{remark}
\begin{proof} Due to Facts \ref{gpoidfromwitness}(2) and \ref{bdgpiso}, even if we replace $f_{01}$ by any $g \in X_{b_{01}}$, we obtain the same
$Y_{01}$.
\end{proof}

\begin{lemma}\label{uniformaction}
A set $C=\{c_i\}_i$ of realizations of $p$ with $b_0\indo C$, and
$g_i\in X_{b_0c_i}$ are given. Then for $\sigma\in X_{b_0}$,
there is an automorphism $\mu=\mu_{\sigma}$ of $\CM$ fixing each
$\ov{c_i}$ and $\ov{b_0}$ pointwise and $\mu(g_i)=g_i.\sigma$. Similarly, if  $D=\{d_i\}_i(\indo b_0)$ is a set of realizations of
$p$ and $h_i\in X_{d_ib_0}$, then there is an automorphism $\tau$
fixing $\ov{d_i}$ and $\ov{b_0}$ such that $\tau(h_i)=\sigma.h_i.$
\end{lemma}
\begin{proof} Take $d\models p$ independent from $b_0C$; and take $h\in
X_{b_0d}$. For each $i$, there is  $h_i\in X_{dc_i}$ such that
$g_i=h_i.h$. Now by stationarity we have $g_{0}\equiv_{\ov{b_0},
\ov{Cd}} g_0.\sigma$ witnessed by some automorphism $\mu$ sending
$g_0$ to $g_0.\sigma$ and fixing $\ov{b_0}, \ov{Cd}$ pointwise. Then
$\mu(g_i)=\mu(h_i.h)=h_i.\mu(h)$ since $h_i\in \ov{Cd}$. Now there
is unique $\tau\in X_{b_0}$ such that $\mu(h)=h.\tau$. Thus
$\mu(g_0)=g_0.\sigma=h_0.h.\tau$. Hence $\sigma=\tau$. Similarly
there is $\tau_i\in X_{b_0}$ such that $\mu(g_i)=g_i.\tau_i$, and
then $\mu(g_i)=g_i.\tau_i=h_i.h.\sigma$. Hence $\tau_i=\sigma$, so
$\mu(g_i)=g_i.\sigma$ as desired. The second clause can be proved similarly.
\end{proof}

Now consider $F_{b_{01}}=F_{01}:=\Aut(Y_{01}/\ov{b_0})$.
%where $Y_{01}$ is defined above.

%[Define $F$ (canonically iso. to $F_{01}$  in the similar manner in/after 3.19 of poly-paper which uniformly
%(regurally) acts on any form of $Y_{ab}$.)??]

\begin{claim}\label{centergp}
\be
\item $X_{b_{01}}\subseteq Y_{01}\subseteq \ov{b_{01}}$.

\item The action of $F_{01}$ on $Y_{01} $ (obviously by $\sigma(g)$ for $\sigma\in
F_{01}$ and $g\in Y_{01}$) is regular (so $|F_{01}|=|Y_{01}|$ but
can be infinite). Hence given $\mu\in G_{01}:=G_{b_{01}}$, there is
its unique extension in $F_{01}$ (we may identify those two).  Thus
$Y_{01}$ is $b_{01}$-invariant.

\item If we let  $(f_0,g_0)\sim(f_1,g_1)\in Y_{01}^2$ when there is (unique) $\sigma\in F_{01}$ such that $g_j=\sigma(f_j)$ ($j=0,1$), then  $\sim$ is an equivalence relation on  $Y_{01}^2$ and
 the map  $[\ ]:Y_{01}^2/\sim\to F_{01}$ sending $[(f_j,g_j)]$ to $\sigma$ is the unique bijection such that
for
$f,g,h\in Y_{01}^2$,
$$[(g,h)]\circ [(f,g)]=[(f,h)].$$

\item  $G_{01}\leq Z(F_{01}).$ Hence for $f,k\in Y_{01}$ and $\sigma\in G_{01}$, we have $f,\sigma(f)\equiv_{\ov{b_0}}k,\sigma(k)$; and
for any $\mu\in F_{01}$, $b'=\mu(b_1)$ and $b_1$ are interdefinable.

\item Suppose that  $\tau\in F_{01}$ and $f,g\in X_{b_{01}}$, so for unique  $e\in
X_{b_0}$ and $\sigma\in G_{01}$, we have $g=\sigma(f)=f.e$. Then
$f,\tau(f)\equiv_{\ov{b_0}} g,\tau(g)$;
$\tau(g)=\tau(f.e)=\tau(f).e$; and
$\sigma(f,\tau(f))=(f.e,\tau(f).e)$.
\ee
\end{claim}
\begin{proof} (1)  is clear.

(2) comes from the fact that for any  $g_0,g_1\in Y_{01}$, they are
interdefinable over $\ov{b_0}$, and $Y_{01} \subseteq \dcl(g_i\bar
b_0)=\dcl(f_{01}\bar b_0)$. Hence from (1),
 it follows $G_{01}$ is
a subgroup of $F_{01}$. The rest clearly follows.

(3) comes from (2), particularly the regularity of the action.

(4) Suppose  $\sigma\in G_{01},\tau\in F_{01}$ are given.  Let
$g=\sigma(f_{01})=f_{01}.\sigma_0$  for some $\sigma_0\in X_{b_0}$, and
let $h=\tau(f_{01})$. Then
$\tau(g)=\tau(f_{01}.\sigma_0)$, and
since $\tau$ fixes $\ov{b_0}$,  $=\tau(f_{01}).\sigma_0=h.\sigma_0.$ Now
by Lemma \ref{uniformaction}, there is an automorphism
fixing $\ov{b}_0$  and sending $(f_{01},h)$ to $(f_{01}.\sigma_0,h.\sigma_0)$, so
$(f_{01},h)\equiv_{\ov{b}_0} (f_{01}.\sigma_0,h.\sigma_0)=(g,\tau(g))$. But since $h\in \dcl(f_{01},\ov{b_0})$ and $g=\sigma(f_{01})$,
we must have
that $\sigma (h)=\tau(g)$, so $=\sigma\circ\tau(f_{01})=\tau\circ\sigma(f_{01})$.
Then due to regularity, we conclude $\sigma\in Z(F_{01})$.

Hence if $k=\mu(f)$ for some $\mu\in F_{01}$, then $\mu(f,\sigma(f))=(k,\sigma(k))$,
in particular $f,\sigma(f)\equiv_{\ov{b_0}}k,\sigma(k)$ \  (*).
 Now if an automorphism $\sigma'$ fixes $b_1\ov{b_0}$ then clearly we can assume
 $\sigma'\in G_{01}$. Then by (*), $\sigma'$ fixes $b'$ too.
Similarly it follows $b_1\in \dcl(b'\ov{b_0})$.
Then due to stationarity it too follows $\dcl(b_1)=\dcl(b')$.

(5) %There is $\sigma \in G_{01}$ such that $g=\sigma(f)$.
%Without loss of generality we can assume $g=f.e$.
  Due to (4), $\sigma(f,\tau(f))=(g,\tau(g))$.
Hence $f,\tau(f)\equiv_{\ov{b_0}} g,\tau(g)$.
Now since $f$ fixes $\ov{b_0}\supseteq X_{b_0}$, particularly it fixes $e$. Hence
$\tau(f.e)=\tau(f).\tau(e)=\tau(f).e$, and the last one follows too.
\end{proof}

But  for $\tau\in F_{01}$ and $f,g\in Y_{01}$, in general $\tp(f,\tau(f))\ne \tp(g,\tau(g))$ (in contrast  to Claim \ref{centergp}(4),(5) and Fact \ref{bdgpiso}(2),(3)). Neither needs $G_{01}$  be equal to $Z(F_{01})$ (see Example \ref{noncentergp}).
\medskip

{\bf For the rest of the paper we fix independent $a,b\models p$ and $f_{ab}\in X_{ab}$.}
We define $Y_{ab}$ just like $Y_{01}$ but with $b_{01},f_{01}$ being replaced by $ab,f_{ab}$.

\begin{lemma}\label{wdfn}
Let $c\models p$ and $c\indo ab$. Let $g\in X_{ca}$. Then for $f\in
Y_{ab}$, it follows $h=f.g\in Y_{cb}$. Moreover  for $h_0=f_{ab}.g$,
we have
 $$h_0f_{ab}\equiv_{\ov{ac}}hf\mbox{ and }f_{ab}f\ov a\equiv_{\ov{b}} h_0h\ov
 c.$$
\end{lemma}
\begin{proof}
Note that $h_0\in X_{cb}$. By stationarity, there is a
$\ov{ca}$-automorphism $\mu$ such that  $\mu(f_{ab})=f$. Then
$\mu(h_0)=\mu(f_{ab}.g)=\mu(f_{ab}).\mu(g)=f.g=h\in\ov{cb}$. We want
to see that $h,h_0$ are interdefinable over $\ov{c}$. Suppose not
say there is $h'\equiv_{\ov{c}h_0}h$ and $h'\ne h$. Then again by
stationarity there is a $\ov{ca}$-automorphism $\tau$ such that
$\tau(h_0h)=h_0h'$. Then for $f=h.g^{-1}$ and $f':=h'.g^{-1}$,
we have $f\ne f'$ but $\tau(f_{ab},f)=\tau(h_0.g^{-1},h.g^{-1})
=(h_0.g^{-1},h'.g^{-1})=(f_{ab},f'),$ a contradiction. Similarly one
can show that $h_0\in\dcl(\ov{c}h)$. Hence $h\in Y_{cb}$.

Now $\mu$
witnesses $h_0f_{ab}\equiv_{\ov{ac}}hf$. To show $f_{ab}f\ov
a\equiv_{\ov{b}} h_0h\ov c$, choose $d(\models p)\indo abc$. Now for
$k_0\in X_{db}$,  by our proof there is $k\in Y_{db}$ such that
$f=k.(k^{-1}_0.f_{ab})$. Then
$h=k.k^{-1}_0.(f_{ab}.g)=k.k^{-1}_0.h_0.$ Now by stationarity,
$f_{ab}\ov a\equiv_{\ov{bd}}h_0\ov c$. Since $k,k_0\in \ov{bd}$, as
desired $f_{ab}f\ov a\equiv_{\ov{bd}} h_0h\ov c$.
\end{proof}

Now we start to construct the new groupoid mentioned. Our first
approximation of  $\mor_{\CF}(a,b)$ is $Y_{ab}$. Beware that
$Y_{ab}(\supseteq X_{ab})$ need not be definable  nor
type-definable. It is just an $ab$-invariant set. So our groupoid
$\CF$ will  only be invariant, and it will be definable only under additional hypotheses (e.g. $\omega$-categoricity).

We recall the binding group $G$ acting on $\CG$ as described in Remark \ref{bdgp}.
 The action need {\em not} be a structure automorphism,
since for $\bsigma\in G$,  in general  $\id_a\not \equiv \bsigma.\id_a\in X_a$, but it is so for $f\in X_{ab}$  (or more generally as in Lemma
\ref{uniformaction} above).
As pointed out in Fact \ref{bdgpiso}(3), there is the group isomorphism  $\rho_{ab}:G\to G_{ab}$ such
that $\rho_{ab}(\bsigma)(f)=\bsigma. f$ for any $f\in X_{ab}$.
We write
$\sigma_{ab}$ for $\rho_{ab}(\bsigma)$. But  when there is no chance
of confusion, we use $\sigma$ for both $\bsigma\in G$ and
$\sigma_{ab}\in G_{ab}$. Moreover, $\sigma_a$ denotes the unique element
in $\bsigma\cap X_a$.
Hence for $f\in X_{ab}$, $\sigma(f)=\sigma.
f=\sigma_b.f=f.\sigma_a$

\begin{remark}\label{tpreserv}
For $\sigma\in G$, and $f\in X_{ab}$ and $g\in X_{cd}$ with
$cd\equiv ab$,  since $G\subseteq \acl(\emptyset)$ we have $f,\sigma. f\equiv g,\sigma. g$.
\end{remark}

Now let $F_{ab}:=\Aut(Y_{ab}/\ov{a})$. Then as in Claim
\ref{centergp}.(2), $G_{ab}\leq F_{ab}$. As just said for any
$cd\equiv ab$, there is the canonical isomorphism between
$\rho_{cd}\circ\rho^{-1}_{ab}:G_{ab}\to G_{cd}$.
 We somehow try to find the
canonically extended isomorphism between $F_{ab}$ and $F_{cd}$ as
well. We do this as follows. Fix an enumeration of  $Y_{ab}=\{g_i\}_i\cup \{g'_j\}_j$ such that
$X_{ab}=\{g_i\}_i$, (and {\em the rest construction depends on this}).
Let $Y_{cd}=\{h_i\}_i\cup \{h'_j\}_j$ such that $X_{cd}=\{h_i\}_i$ and
 $\la
g_i\ra^\frown \la g'_j\ra ab\equiv \la h_i\ra^\frown \la h'_j\ra
cd$.
 Now due to regularity of the  action, for each $i$ or $j$ there
is unique $\mu^{ab}_i$ or $\mu^{ab}_j\in F_{ab}$ such that
$\mu_i(g_0)=g_i$ or $\mu_j(g_0)=g'_j$. Similarly we have
$\mu^{cd}_i$ or $\mu^{cd}_j\in F_{cd}.$

\begin{claim}
The correspondence $\mu^{ab}_i\mapsto\mu^{cd}_i$ or
$\mu^{ab}_j\mapsto\mu^{cd}_j$ is a well-defined isomorphism from
$F_{ab}$ to $F_{cd}$ extending $\rho_{cd}\circ\rho^{-1}_{ab}$.
\end{claim}
\begin{proof}
Assume $\{k_i\}_i\cup \{k'_j\}_j$ is another arrangement of $Y_{cd}$
such that $\la k_i\ra^\frown \la k'_j\ra \equiv_{cd}\la
h_i\ra^\frown \la h'_j\ra$. Then $k_0=\sigma(h_0)$ for some
$\sigma\in G_{cd}$. Thus by Claim \ref{centergp}, we have
$\sigma(h_0,\mu^{cd}_i(h_0))=(k_0,\mu^{cd}_i(k_0))$ and so
$h_0,\mu^{cd}_i(h_0)\equiv_{\ov c}k_0,\mu^{cd}_i(k_0)$. Then due to
interdefinability, we must have $\mu^{cd}_i(k_0)=k_i$. Similarly
$\mu^{cd}_j(k_0)=k'_j$. Hence the map is well-defined. It easily
follows that the map in fact is an isomorphism.  Moreover due to
\ref{tpreserv} we  see that it  extends
$\rho_{cd}\circ\rho^{-1}_{ab}$.
\end{proof}

Hence now we  fix an {\em extended binding group} $F\geq G$
isomorphic to $F_{01}$. (Contrary to $G\subseteq \acl(\emptyset)$, $F$ {\em need not} live in $\acl(\emptyset)$.) Then there is a canonical isomorphism
$\rho^F_{cd}:F\to F_{cd} $ extending $\rho_{cd}$ in such a way that
$\rho^F_{cd}\circ(\rho^F_{ab})^{-1}$ is the correspondence defined
above. Now for $\bmu\in F$, we use $\mu_{ cd}$ or simply $\mu$ to
denote $\rho^F_{cd}(\bmu)$. Note that a mapping $\bmu.
f:=\mu_{cd}(f)$ is clearly a regular action of $F$ on $Y_{cd}$
extending that of $G$ on $X_{cd}$.

\begin{claim}
If $cd\indo a$, then for $f\in X_{cd}, g\in X_{ac}$, we have
$\bmu. (f.g)= (\bmu.f).g$.
\end{claim}
\begin{proof}
This follows from Lemma \ref{wdfn}.
%Clearly $\mu(f.g)=\mu(k_d).k^{-1}_d.f.g=\mu(f).g.$
\end{proof}

Assume now $c(\models p)\indo ab$, and $g\in Y_{ab}, h\in Y_{bc}$ are
given. We want to define a composition $h.g\in Y_{ac}$ extending
that for $\CG$. Note now $g=\tau_0(g_0)$ and $h=\sigma_0(h_0)$ for
some $\mbox{\boldmath $\tau_0$}, \mbox{\boldmath $\sigma_0$}\in F$
and $g_0\in X_{ab}, h_0\in X_{bc}$. We define $h.g:=
(\mbox{\boldmath $\sigma_0$}\circ \mbox{\boldmath
$\tau_0$}).(h_0.g_0)=\sigma_0\circ\tau_0(h_0.g_0).$

\begin{claim}\label{compost}
The composition map is well-defined, invariant under any ($A$-)automorphism of $\CM$, and extends that of $\mor(\CG)$. For any $f\in Y_{ac}$,
there is unique $h'\in Y_{bc}$ ($g'\in Y_{ab}$, resp.) such that $f=h'.g$ ($f=h.g'$ resp.).
\end{claim}
\begin{proof}
Let $g=\tau_1(g_1)$ and $h=\sigma_1(h_1)$ for some $\mbox{\boldmath
$\tau_1$}, \mbox{\boldmath $\sigma_1$}\in F$ and $g_1\in X_{ab},
h_1\in X_{bc}$. Then since $\sigma^{-1}_0\circ\sigma_1(h_1)=h_0$
and $\tau^{-1}_0\circ\tau_1(g_1)=g_0$, due to uniqueness we have
that both $\mbox{\boldmath $\sigma^{-1}_0$}\circ \mbox{\boldmath
$\sigma_1$}$, $\mbox{\boldmath $\tau^{-1}_0$}\circ \mbox{\boldmath
$\tau_1$}$ are in $G$ so in the center of $F$. Now due to Lemma  \ref{wdfn},
$$\begin{array}{cllll}
\sigma_0\circ\tau_0(h_0.g_0)&=
&\sigma_0\circ\tau_0\circ\sigma^{-1}_0\circ\sigma_0(h_0.g_0)&=&
\sigma_0\circ\tau_0\circ\sigma^{-1}_0(\sigma_0(h_0).g_0)\\
&=&\sigma_0\circ\tau_0\circ\sigma^{-1}_0(\sigma_1(h_1).g_0)&=&
\sigma_0\circ\tau_0\circ(\sigma^{-1}_0\circ\sigma_1)(h_1.g_0)\\
&=&\sigma_1\circ\tau_0(h_1.g_0)&=
&\sigma_1\circ\tau_1\circ(\tau^{-1}_1\circ\tau_0)(h_1.g_0)\\
&=&\sigma_1\circ\tau_1(h_1.(\tau^{-1}_1\circ\tau_0)(g_0))&=&
\sigma_1\circ\tau_1(h_1.(\tau^{-1}_1(\tau_1(g_1))))\\
&=&\sigma_1\circ\tau_1(h_1.g_1).& &
\end{array}$$
Automorphism invariance clearly follows from the same property for
$\mor(\CG)$ and the choice of the isomorphism $\rho^F_{ab}$. Moreover
by taking $\tau_0=\sigma_0=\id$, we see that the composition clearly
extends that for $\CG$. Lastly $f=\tau(f_1)$ for some $f_1\in
X_{ac}$. Now there is $h'_1\in X_{bc}$ such that $f_1=h'_1.g_1.$
Put $h'=\tau\circ\tau_1^{-1}(h'_1)$. Then by the definition,
$f=(\tau\circ  \tau^{-1}_1)\circ\tau_1(h_1'.g_1)=h'.g$. For any $h''(\ne h')\in Y_{bc}$ it easily follows that
$f\ne h''.g$. Hence $h'$ is unique such element.
\end{proof}

The rest of the construction of $\CF$ will be similar to that of $\CG$ in
\cite{GK}. $\mbox{Ob}(\CF)$ will be the same as
$\mbox{Ob}(\CG)=p(\CM)$. Now for arbitrary $c,d\models p$, an {\em
$n$-step directed path} from $c$ to $d$ is a sequence
$(c_0,g_1,c_1,g_2...,c_n)$ such that $c=c_0, d=c_n, c_{i-1}c_i\equiv
ab$ and $g_i\in Y_{c_{i-1}c_i}$. Let $D^n(c,d)$ be the set of all
$n$-step directed paths.  For $q=(c_0,g_1,c_1,g_2...,c_n)\in
D^n(c,d)$ and $r=(d_0,h_1,d_1,h_2...,d_m)\in D^m(c,d)$ we say they
are equivalent (write $r\sim s$) if for some  $c^*(\models p)\indo
qr$ and $g^*\in Y_{c^*c}$, we have $g^*_n=h^*_m\in Y_{c^*d}$ where
$g^*_0=h^*_0=g^*$ and  $g^*_{i+1}=g_{i+1}.g_i^*$ ($i=0,...,n-1$) and
$h^*_{j+1}=h_{j+1}.h_j^*$  ($j=0,...,m-1$). Due to stationarity the
 relation is independent from the choices of $c^*$ and
$g^*$, and is an equivalence relation. Similarly to Lemma
\cite[2.12]{GK}, one can easily see using Claim \ref{compost} that
for any $q\in D^n(c,d)$, there is $r\in D^2(c,d)$ such that $q\sim
r$. Then $D^2(c,d)/\sim$ will be our $\mor_{\CF}(c,d)$, and
composition will be concatenation of paths. The identity morphism in
$\mor_{\CF}(c,c)$ can be defined just like in \cite[2.15]{GK}.
Now our groupoid $\CF$ is clearly  connected, and it extends $\CG$ (see Proposition \ref{compextn}). An argument similar to that in
\cite[2.14]{GK} implies there is a canonical $ab$-invariant  1-1 correspondence between
$Y_{ab}$ and $\mor_{\CF}(a,b).$ Indeed the same argument shows that for any $c,d\models p$
(not necessarily independent), there too exists a canonical injection from  $X_{cd}$ to $\mor_{\CF}(c,d).$

Now for $f\in Y_{ab}$ (or $\in X_{cd}$, resp.), in the rest {\em we
let $\ul{f}$  denote the
corresponding element in} $\mor_{\CF}(a,b)$ (or $\mor_{\CF}(c,d)$, resp.).
Then similarly to $Y_{ab}$, it follows
$$\mor_{\CF}(a,b)=\{ x\in \dcl(\ul{f},\ov{a})|\ x\equiv_{\ov{a}} \ul{f} \mbox{ and }\dcl(x\ov{a})=\dcl(\ul{f}\ov{a})\}\subseteq \ov{ab}.$$
But $\CF$ need not be definable nor
type-definable nor hyperdefinable. It is just an invariant groupoid.
\smallskip

As pointed out in \ref{centergp},   $Y_{ab}$ is $ab$-invariant. Now
if it is type-definable then as it is a bounded union of definable
sets, by compactness it indeed is definable and a finite set. (This
happens when $T$ is $\omega$-categorical.) For this case let us add
a bit more explanations that are not explicitly mentioned in
\cite{GK}. By compactness now, $\sim$ turns out to be definable:
Note that $D^2(p):=\bigcup \{D^2(c,d)|\ c,d\models p\}$ is
$\emptyset$-type-definable. Then there clearly is an
$\emptyset$-definable equivalence relation $E$ on $D^2(p)$ each of
whose class is of the form $D^2(c,d)$. In each $E$-class, there are
exactly $|Y_{ab}|$-many $\sim$-classes. Hence $\sim$ is
$\emptyset$-definable relatively on $D^2(p)$ as well. Hence $[r]\in
\mor_{\CF}(c,d)$ is an imaginary element and the maps $[r]\mapsto c$
or $d$ (the first and last components of $r$) are
$\emptyset$-definable  $\init,\ter$ maps, respectively. Similarly the composition map of morphisms is
$\emptyset$-definable.
Therefore $\CF$ is a (relatively)
$\emptyset$-definable groupoid.
\smallskip

We return to the general context of the invariant $\CF$. For notational simplicity,
use ${\ul{Y}}_{cd}$ to denote $\mor_{\CF}(c,d)$, and use ${\ul{Y}}_{c}$ for
${\ul{Y}}_{cc}$.   We  state some observations
regarding $\CF$.
%Now for $f\in Y_{ab}$, we write $\ul{f}$ for its canonically
%corresponding element in $\Phi_{ab}$.

\begin{remark}\be\item Note  that for $\sigma\in
F_{ab}$ and $f\in Y_{ab}$, we have $\sigma(f)\in Y_{ab}$ and both
$\ul{f}, \ul{\sigma(f)}\in \ul{Y}_{ab}$. However  depending on context, $\sigma(\ul f)$ may be
 in $\ul{Y}_{ab}$, or $\ul{Y}_{a\sigma(b)}.$
For $f,g\in Y_{ab}$, clearly $f\ul{f}\equiv g\ul{g}$. Hence in this sense obviously
$\sigma(\ul f)\in \ul{Y}_{ab}.$ But since $\sigma(f)$  is in $Y_{a \sigma(b)}$ too,
we can say $\sigma(\ul{f})\in \ul{Y}_{a\sigma(b)},$ as well. To remove this confusion,
one may put a prefix   $ab$ to any $\ul{f}\in \ul{Y}_{ab}$. But instead, in the rest
we only regard  $\sigma(\ul f)\in \ul{Y}_{ab}.$

\item We know that $\ul{Y}_{ab}\subseteq \ov{ab}$.
For any $c\models p$, we too have $\ul{Y}_{bc}\subseteq \ov{bc}$: We can assume $a\indo bc$.
Now let $f\in \ul{Y}_{bc}$. Suppose that $f\not \in \ov{bc}$, and let  $\{f_i\}$ be a set of  infinitely many
conjugates of $f$ over $\ov{b}\ov{c}$. Now due to stationarity, we can then assume that
 all $f_i$'s have the same type over $\ov{c}\cup \ov{ba}$. Hence given  $x\in \ul{Y}_{ab}$, all $f_i.x\in \ul{Y}_{ac}$ have the same type over $ac$, contradicting $f_i.x\in \ov{ac}$.
\ee
\end{remark}

We get now the following results for $\CF$ similarly to those of
$\CG$.

\begin{proposition}\label{isomorphic}
The group $F_{ab}$ is isomorphic to $\ul{Y}_a$. In fact for any $\sigma\in\ul{Y}_b$, there is
$\sigma_b\in F_{ab}$  such that for any $f\in
Y_{ab}$, $\ul{\sigma_b(f)}=\sigma.\ul f$. Hence $F_{ab}=\{\sigma_b|\ \sigma\in
\ul{Y}_b\}$.
\end{proposition}
\begin{proof}
The proof will be similar to that of Claim \ref{iso}. Define a map
$\eta:\ul{Y}_b\to F_{ab}$ such that for $\sigma\in \ul{Y}_b$ and any $f\in
Y_{ab}$, we let $\eta(\sigma)(f)=g$ where $\ul g=\sigma.\ul f$. Hence due to that $f\ul{f}\equiv g\ul{g}$,
we have
$\eta(\sigma)(\ul{f})=\ul{g}$ too.
Then
$\eta$ is a well-defined, since if $h\in
Y_{ab}$, then there is $x\in \ul{Y}_a\subseteq \ov{a}$ such that $\ul{h}=\ul{f}.x$, and  thus
$\eta(\sigma)(\ul{h})=\eta(\sigma)(\ul{f}).x=\ul{g}.x=\sigma.\ul{f}.x=\sigma.\ul{h}.$

Moreover clearly $\eta$ is  1-1 and onto  since
any $\mu\in F_{ab}$ is determined by $(f_{ab},\mu(f_{ab}))$.
It is obvious   $\eta$ is in fact a group
isomorphism. Now we take $\sigma_b=\eta(\sigma)$.
\end{proof}

\begin{proposition}\label{compextn}
For $c(\models p)\indo ab$ and $f\in Y_{ab},g\in Y_{bc}, h\in
Y_{ac}$, we have $h=g.f$ (the composition map is  defined before Claim \ref{compost}) iff $\ul h=\ul{g}.\ul{f}$. Moreover, $\CF$ extends the composition of $\CG$.
\end{proposition}
\begin{proof}
 Since the composition relation defined in
\ref{compost} is invariant relation, we can find an $\emptyset$-invariant relation
$\theta(x,y,z)$ such that for any $a'b'c'\equiv
abc$ and $f'\in Y_{a'b'},g'\in Y_{b'c'}, h'\in Y_{a'c'}$, we have
$h'=g'.f'$ iff $\theta(a'b'f',b'c'g',a'c'h')$ holds. Then the rest
proof of the proposition will be exactly the same as that of
\cite[2.19]{gkk}, hence we omit it.

We now step by step show that $\CF$ extends the composition of $\CG$.
 Let $c\models p$ be given such that $ac\indo b$.  Let $y\in X_{ab},z\in X_{bc}$.
  %z\in X_{ac}$ such that $z=y.x$ (in $\CG$).
Choose
$a'(\models p)\indo abc$, and $ x\in X_{a'a}.$ Then $z.y.x\in X_{a'c}$, and by the definitions of the concatenating composition and the
injection from $X_{ac}$ to $\ul{Y}_{ac}$, we have $\ul{z}.\ul{y}=\ul{z.y}$ in $\CF$.

Now more generally let $d\models p$ be given such that $acd\indo b$. Let $s\in X_{ac}$, $t\in X_{cd}$.
Choose $u\in X_{ab}$. Then there are $v\in X_{bc}$ and $w\in X_{bd}$ such that $s=v.u \in X_{bc}$, and  $w.v^{-1}=t$ (in $\CG$). Then by the previous argument, in $\CF$, we have
$ \ul{t}.\ul{s}=\ul{w.v^{-1}}.\ul{v.u}=\ul{w}.\ul{v^{-1}}.\ul{v}.\ul{u}=\ul{w}.{\ul{v}}^{-1}.\ul{v}.\ul{u}=
\ul{w.u}=\ul{t.s}.$
\end{proof}

We now give an example where $X_a$ is a proper subgroup of  $Z(\ul{Y_a})$.

\begin{example}\label{noncentergp}
Consider the same example $(O,M,.,\init,\ter)$ as in section 2, but where the binding group $G$ is abelian. Namely it is a connected finitary abelian groupoid.
We add one more sort $I$ and an equivalence relation $E$ on $I$ such that each $E$-class has $2$ elements, and there also is a projection function
$\pi_E:I\to O$  (all are in the language) so that $O=I/E$. We let $N$ be the resulting extended structure. Now choose $c\ne d\in O$, $f\in \mor(c,d)$, and let $\{c_0,c_1\}=\pi^{-1}(c)$, and $\{d_0,d_1\}=\pi^{-1}(d)$.
Now clearly $(c_0c_1c, d_0d_1d,\dots,f,\dots,.(x,yz))$ where $.(x,y,z)$ is a formula indicating the composition $z=y.x$, can be considered as  a symmetric witness. Let $\CG$ be the abelian groupoid obtained from the symmetric witness. Then clearly $X_{c_0c_1c}$ is isomorphic to $G$. Let $\CF$ be the groupoid as we constructed above from $\CG$. Then since an automorphism of $N$ can swap   $d_0,d_1$  while fixing $c_0c_1cd$, we have that $G$ is central in  $F:=\ul{Y}_{c_0c_1c}$, while $G$ has only two conjugates in $F$. Then it easily follows that $F$ is abelian too, so
$G \lneq Z(F)=F$.
\end{example}

\section{Approximation of the non-commutative groups}

In this last section we discuss a possible limit of  the vertex groups of the non-commutative groupoids we have constructed in previous section.
As before, we keep suppressing $A=\acl(A)$ to $\emptyset$.
 Fix a complete type $q$ (of possibly infinite arity) over $\emptyset$. Choose independent
 $u,v,w\models q$. Recall that
 $$\Gamma_2(q):=\Aut(\widetilde{uv}/\ov{u},\ov{v}),$$
 where $\widetilde{uv}:=\ov{uv}\cap \dcl(\ov{uw},\ov{vw})$.

The following fact is simply a restatement of Proposition \ref{full symm} and Fact \ref{gpoidfromwitness}.

 \begin{fact}\label{dirsys}
 Let a finite tuple $f\in \widetilde{uv}\smallsetminus \dcl(\ov{u},\ov{v})$ be given.
 Then there are  a generic abelian
  groupoid $\CG'$ in  $q'\in S(\emptyset)$ of  finite arity; and independent $u',v'\models q'$ with $f'\in \mor_{\CG'}(u',v')$ such that
  $f\in\dcl(f')$, and $u'\subseteq \ov{u}$, $v'\subseteq \ov{v}$.
\end{fact}

We let

$\begin{array}{ll}
I=I_q:=\{ f & \in \widetilde{u_f, v_f}:\ \CG_f  \mbox{ is a generic  abelian groupoid in }
\tp(u_f)=\tp(v_f) \\
& \mbox{ such that }
f\in \mor_{\CG_f}(u_f, v_f)\mbox{ with independent finite }\\
& \mbox{tuples } u_f(\subseteq \ov{u}), v_f(\subseteq \ov{v})\}.
\end{array}$

\medskip

\noindent On the other hand we let,

$\begin{array}{ll}
J=J_q:=\{ f  & \in \widetilde{u_f, v_f}:\ \CF_f  \mbox{ is a non-commutative groupoid obtained from  }\\
& \mbox{the generic abelian groupoid } \CG_f \mbox{ in } \tp(u_f)=\tp(v_f)  \mbox{ such that }\\
& f\in \mor_{\CF_f}(u_f, v_f)\mbox{ with independent finite tuples }
 u_f(\subseteq \ov{u}), v_f(\subseteq \ov{v})\}.
\end{array}$

\medskip
\noindent
For $f\in I_q$,  we write
 $G_f$  to denote   $G_{u_fv_f}$ as in section 3.
Now by Fact \ref{dirsys},   $(I,\leq_I)$ with letting $f\leq_I f'$ iff  $f\in \dcl(f')$,
 $\init(f)\in\dcl(\init(f'))$, and $\ter(f)\in\dcl(\ter(f'))$ is a direct system.

Now for $f\leq_I f'\in I$,  any $\sigma'\in G_{f'}$
 fixes $u_{f'}v_{f'}$ pointwise.
 Hence $\sigma'(f)\in X_{u_fv_f}$, and
we write $(\sigma' \restriction f)$ to denote the unique $\sigma\in G_{f}$ such that
$\sigma (f)=\sigma'(f)$; and  $\chi^{f'}_{f}:G_{f'}\to G_f$ to denote
 the group
homomorphism  sending $\sigma'$ to
$(\sigma' \restriction f)$.
Due to stationarity it indeed is an epimorphism.
 Clearly $\chi^f_f=\id$. Then
 $$\CS_{I}:=(\{G_f|\ f\in I\}, \{\chi^{f'}_{f} |\ f\leq_I f'\in I\})$$ forms a directed system of finite abelian groups. As pointed out in \cite[Theorem 2.25]{gkk}, the inverse limit of $\CS_{I}$ is isomorphic to
  $\Gamma_2(q)$, so it is a  profinite abelian group.

However for $\{F_f| \ f\in J\}$ where $F_f:=F_{u_fv_f}$ as in section 3, it is not clear how to   give an order relation and  transition maps to make this a directed system of groups.
There are a couple of obstacles to do this. For example, in general given an elementary map $\sigma$ of $\CM$, and a tuple $cd$, even if
$cd$ and $\sigma(cd)$ are interdefinable, $c$ and $\sigma(c)$ need not be so, and vice versa. But we can consider {\em partial} transition maps among $F_f$'s and their limit as follows. We let
$$\Pi_2(q):=\{ \sigma \in \aut(\widetilde{uv}/\ov{u}): \mbox{ for any } f\in J, \ \dcl(f\ov{u})=\dcl(\sigma(f)\ov{u})\}.$$
For $f\in J$ and $\sigma \in \Pi_2(q)$, we similarly write
$(\sigma\restriction f)$ to denote the unique $\sigma'\in F_f$ such that $\sigma(f)=\sigma'(f)$.
We let $$\Pi_f:=\{(\sigma\restriction f)| \ \sigma\in \Pi_2(q)\}.$$

\begin{proposition}  The following hold.
\be
\item
$\Pi_2(q):=\{ \sigma \in \aut(\widetilde{uv}/\ov{u}): \mbox{ for any } f\in I, \ \dcl(f\ov{u})=\dcl(\sigma(f)\ov{u})\}$.
\item
%for any $\sigma \in \Pi_2(q)$, and $f\in J$, $\sigma\restriction f$ (defined as ....) belongs to $F_f$ and
For $f\in J$, we have
 $G_f\leq \Pi_f\leq F_f$. Now
$$\CS_{J}:=(\{\Pi_f|\ f\in J\}, \{\chi^{f'}_{f} |\ f\leq_J f'\in J\})$$ forms a directed system of groups, where $\leq_J$ and $\chi^{f'}_f$ are similarly defined as in $\CS_I$. Moreover $\Pi_2(q)$ is the inverse limit of $\CS_J$.

\item
$\Gamma_2(q)\leq Z(\Pi_2(q))$.
\item Both $\Gamma_2(q)$ and 
$\Pi_2(q)$ are  normal subgroups of $\aut(\widetilde{uv}/\ov{u})$.
\ee
\end{proposition}
 \begin{proof}
 (1) Clear (see Claim \ref{centergp}(4)).

 (2) That $\Pi_f\leq F_f$ is clear by definition, and that  $G_f\leq \Pi_f$ is also clear since
 $\Gamma_2(q)\leq \Pi_2(q)$. Now since every  automorphism in $\Pi_f$ is the restriction of that in $\Pi_2(q)$, it follows that $\CS_J$ forms a directed system of groups with the transition maps $\chi^{f'}_f$. The rest proof of (2) is standard. Let $\Pi$ be the inverse limit of $\CS_J$.
We define a homomorphism $\varphi:\Pi_2(q)\to \Pi$ by sending $\sigma \in \Pi_2(q)$ to the element in $\Pi$
  represented by the function $f\in J \mapsto (\sigma\restriction f) \in \Pi_f$. This embedding is obviously one-to-one and due to compactness it is surjective too.

(3) comes from (2) and that $G_f\leq Z(\Pi_f)$.

(4) Let $\sigma\in \Gamma_2(q)$. Then any of its conjugates in $ \aut(\widetilde{uv}/\ov{u})$ fixes $\ov{v}$ pointwise. Hence 
$\Gamma_2(q)\trianglelefteq \aut(\widetilde{uv}/\ov{u})$.

Now let $\sigma \in \Pi_2(q)$ and $\mu\in \aut(\widetilde{uv}/\ov{u})$. Then for any  $f\in I$, we have  $g:=\mu(f)\in I$ too, and $g$ and $\sigma(g)$ are interdefinable over $\ov{u}$.
Hence so are  $f$ and $\mu^{-1}\circ \sigma (g)=\mu^{-1}\circ \sigma \circ \mu(f)$ over $\ov{u}$.
Therefore $\mu^{-1}\circ \sigma \circ \mu \in \Pi_2(q)$ and (4) is proved.
 \end{proof}

\end{document}